\documentclass[12pt]{iopart}

\usepackage{iopams} 
\usepackage{setstack}
\usepackage{graphicx} 

\newtheorem{theorem}{Theorem}[section]
\newtheorem{corollary}{Corollary}[section]
\newtheorem{proposition}{Proposition}[section]
\newtheorem{lemma}{Lemma}[section]
\newtheorem{definition}{Definition}[section]
\newtheorem{remark}{Remark}[section]
\newenvironment{proof}{{\noindent\bf Proof.}\quad}{\hfill $\square$\\}
\eqnobysec

\begin{document}

\title[Tikhonov regularization for polynomial approximation in Gauss points]{Tikhonov regularization for polynomial approximation problems in Gauss quadrature points}

\author{Congpei An$^1$ and Hao-Ning Wu$^2$}

\address{$^1$ School of Economic Mathematics, Southwestern University of Finance and Economics, Chengdu, China}
\address{$^2$ Department of Mathematics, The University of Hong Kong, Hong Kong, China}
\eads{\mailto{ancp@swufe.edu.cn} and \mailto{hnwu@hku.hk}}
\vspace{10pt}
\begin{indented}
\item[]November 2020
\end{indented}

\begin{abstract}
This paper is concerned with the introduction of Tikhonov regularization into least squares approximation scheme on $[-1,1]$ by orthonormal polynomials, in order to handle noisy data. This scheme includes interpolation and hyperinterpolation as special cases. With Gauss quadrature points employed as nodes, coefficients of the approximation polynomial with respect to given basis are derived in an entry-wise closed form. Under interpolatory conditions, the
solution to the regularized approximation problem is rewritten in forms of two kinds of barycentric interpolation
formulae, by introducing only a multiplicative correction factor into both classical barycentric
formulae. An $L_2$ error bound and a uniform error bound are derived, providing similar information that 
Tikhonov regularization is able to reduce the operator norm (Lebesgue constant) and the error term related to the level of noise, both by multiplying a correction factor which is less than one. Numerical examples show the benefits of Tikhonov regularization when data is noisy or data size
is relatively small.
\end{abstract}

%
\vspace{2pc}
\noindent{\it Keywords}: Tikhonov regularization, hyperinterpolation, barycentric interpolation, Gauss quadrature, polynomial approximation.
%

%
%
%

\section{Introduction}\label{sec;introduction}
Polynomial approximation is used as the basic means of approximation in many fields of numerical analysis, such as interpolation and approximation theory, numerical integration, numerical solutions to differential and integral equations. In particular, the orthogonal polynomial expansion occurs and plays an important role in these fields. It has been known that interpolation based on zeros of orthogonal polynomials prevails over that based on equispaced points, and it is widely applied in numerical integration, spectral methods, and so on \cite{trefethen2013approximation}. The central issue in orthogonal polynomial computation is a fact that any nice enough function $f(x)$ can be expanded by a series of orthogonal polynomial \cite{gautschi2004orthogonal,von1939orthogonal,xiang2012error}
\begin{equation}\label{equ:series}
f(x)=\sum_{\ell=0}^{\infty}c_{\ell}\Phi_{\ell}(x),\quad c_{\ell}=\frac{\int_{-1}^1w(x)f(x)\Phi_{\ell}(x){\rm d}x}{\int_{-1}^1w(x)\Phi_{\ell}^2(x){\rm d}x},\quad\ell=0,1,\ldots,
\end{equation}
where $\{\Phi_{\ell}(x)\}_{\ell=0}^{\infty}$ is a family of orthogonal polynomials with respect to a nonnegative weight function $w(x)$ which satisfies 
$\int_{-1}^1w(x){\rm d}x<\infty$, and $\Phi_{\ell}(x)$ is of degree $\ell$. We only talk about approximations on $[-1,1]$ in this paper, as any bounded interval can be scaled to $[-1,1]$. One natural approximation to $f$ in the polynomial space $\mathbb{P}_L$ of degree at most $L$ is the polynomial obtained by \emph{truncation} to degree $L$:
\begin{equation*}
p_L^{\rm{trun}}(x)=\sum_{\ell=0}^Lc_{\ell}\Phi_{\ell}(x),
\end{equation*}
with coefficients $\{c_{\ell}\}_{\ell=0}^L$ are the same as those of $f$ which are given in \eref{equ:series}. Another is the polynomial obtained by \emph{interpolation}:
\begin{equation*}
p_L^{\rm{inter}}(x)=\sum_{\ell=0}^Ld_{\ell}\Phi_{\ell}(x),
\end{equation*}
called \emph{interpolant}, where $\{d_{\ell}\}_{\ell=0}^L$ is a set of coefficients which are determined such that $p_L^{\rm{inter}}(x)$ interpolates some given discrete points. Though $\{c_{\ell}\}_{\ell=0}^L$ can be approximated via some quadrature rules on some discrete points, they are usually different from $\{d_{\ell}\}_{\ell=0}^L$ in general.
 
To establish a connection between coefficients in the truncated polynomial and the polynomial interpolant, and to compute coefficients in concerned expansions efficiently on the computer as well, we consider approximations with coefficients computed in a discrete way and we use normalized orthogonal (orthonormal) polynomials $\{\tilde{\Phi}_{\ell}\}_{\ell=0}^L$. That is, we are interested in approximation of a function (possibly noisy) $f\in\mathcal{C}([-1,1])$, where $\mathcal{C}([-1,1])$ is the space of continuous functions on $[-1,1]$, by a polynomial
\begin{equation}\label{solutionp}
p_{L}(x)=\sum_{\ell=0}^L\beta_{\ell}\tilde{\Phi}_{\ell}(x)\in \mathbb{P}_L,\quad x\in[-1,1],
\end{equation}
where $\{\beta_{\ell}\}_{\ell=0}^L$ is a set of coefficients to be determined. Orthogonal polynomials are normalized as $\tilde{\Phi}_{\ell}(x):={\Phi_{\ell}(x)}/{\|\Phi_{\ell}(x)\|_{L_2}}$, $\ell=0,\ldots,L$, where the $L_2$ norm 
\begin{equation}\label{equ:l2norm}
\|f\|_{L_2}:=\sqrt{\left<f(x),f(x)\right>_{L_2}}=\left(\int_{-1}^1w(x)|f(x)|^2{\rm d}x\right)^{1/2}
\end{equation}
is induced by the $L_2$ inner product $\left<f(x),g(x)\right>_{L_2}:=\int_{-1}^1w(x)f(x)g(x){\rm d}x$ which defines the orthogonality in orthogonal polynomials \cite{gautschi2004orthogonal,von1939orthogonal}. Normalization would not change the final approximation polynomial $p_L$, but it would greatly simplify the explicit expressions and the computation of $\{\beta_{\ell}\}_{\ell=0}^L$, see Section \ref{sec:explicit}.

If the approximation is studied in a discrete way, then the determination of coefficients $\{\beta_{\ell}\}_{\ell=0}^L$ shall depend on data $\{f(x_j)\}$ sampled on $\{x_j\}$. In practice, however, the sampling procedure is often contaminated by noise, and the classical least squares approximation is sensitive to noisy data. Hence we may introduce regularization techniques to handle this case. A widely used regularization technique is the Tikhonov regularization \cite{lu2013regularization,tikhonov1977solutions}, which adds an $\ell_2^2$ penalty. This technique shrinks all coefficients $\{\beta_{\ell}\}_{\ell=0}^L$ towards zero to provide stability and reduce noise. Tikhonov regularization has been widely investigated in inverse and ill-posed problems \cite{lu2009sparse,lu2006analysis,wei2016tikhonov,xiang2013regularization,xiang2015randomized,zhong2012multiscale}, in which some least squares problems are also addressed \cite{lu2010model,lu2010regularized}.

Suppose the size of sampling data is $N+1$, thus our problem with consideration to discrete format and Tikhonov regularization is stated as
\begin{equation}\label{problem:l2regu}
\min_{\beta_{\ell}\in\mathbb{R}}~~ \left\{\sum_{j=0}^N\omega_j\left(\sum_{\ell=0}^{L}\beta_{\ell}\tilde{\Phi}_{\ell}(x_j)-f(x_j)\right)^2+\lambda\sum_{\ell=0}^L|\beta_{\ell}|^2\right\},\quad\lambda>0,
\end{equation}
where $f$ is a given continuous function with values (possibly noisy) taken at a set $\mathcal{X}_{N+1}=\{x_0,x_1,\ldots,x_N\}$ on $[-1,1]$; $\{\omega_0,\omega_1,\ldots,\omega_N\}$ is a set of some weights; and $\lambda>0$ is the regularization parameter.

It is natural to choose a set of zeros of the corresponding orthonormal polynomial $\tilde{\Phi}_{N+1}$ to be the set $\mathcal{X}_{N+1}$, because when the basis for the approximation \eref{solutionp} is chosen as $\{\tilde{\Phi}_{\ell}\}_{\ell=0}^L$, this is a usually adopted choice. Apart from this point, the choice helps us to establish the connection between the approximation polynomial \eref{solutionp} and interpolation, as many efficient interpolation schemes are based on zeros of orthogonal polynomials, for example, Chebyshev interpolation which are based on zeros of Chebyshev polynomials \cite{trefethen2013approximation}, and the fast and stable barycentric interpolation \cite{berrut2004barycentric,wang2014explicit,wang2012convergence}. It is well known that zeros of the orthogonal polynomial $\Phi_{N+1}$ of degree $N+1$ are just $N+1$ Gauss quadrature points \cite{gautschi2011numerical,kress1998numerical}. 

If we require $\{\omega_j\}_{j=0}^N$ to be $N+1$ Gauss quadrature weights, and $L$ and $N$ to satisfy $2L\leq2N+1$, then the first part in the objective function of \eref{problem:l2regu} is the Gauss quadrature approximation  
\begin{eqnarray*}
\sum_{j=0}^N\omega_j\left(\sum_{\ell=0}^{L}\beta_{\ell}\tilde{\Phi}_{\ell}(x_j)-f(x_j)\right)^2
&\approx
\int_{-1}^1w(x)\left(\sum_{\ell=0}^L\beta_{\ell}\tilde{\Phi}_{\ell}(x)-f(x)\right)^2{\rm d}x\\
&=\int_{-1}^1w(x)\left(p_L(x)-f(x)\right)^2{\rm d}x.
\end{eqnarray*}
These requirements are kept in the whole paper. Note that the interval we consider is bounded, hence the orthonormal basis is chosen as normalized Jacobi polynomials, which are defined on $[-1,1]$, from the large family of orthogonal polynomials \cite{gautschi2004orthogonal,von1939orthogonal}.

If Gauss quadrature is adopted, we can construct entry-wise closed-form solutions to problem \eref{problem:l2regu} and show that this regularized approximation scheme is a generalization of hyperinterpolation \cite{sloan1995polynomial}. Under interpolatory conditions, we rewrite the approximation polynomial \eref{solutionp} with constructed coefficients in forms of modified Lagrange interpolation and barycentric interpolation \cite{berrut2004barycentric}, respectively, presenting Tikhonov regularized modified Lagrange interpolation formula \eref{L2mdf} and Tikhonov regularized barycentric interpolation formula \eref{L2regubary}. Tikhonov regularization introduces only a simple factor $1/(1+\lambda)$ into both formulae in their classical versions. We also study the approximation quality of problem \eref{problem:l2regu} in terms of the $L_2$ norm and the uniform norm, respectively, showing operator norms of this kind of approximation can be reduced by multiplying the same factor $1/(1+\lambda)$, and an error term for noise can also be reduced by the factor. Though Tikhonov regularization reduces the above terms, it would introduce an additional error term into the total error bound, which is dependent on the best approximation polynomial $p^*$.

This paper is organized as follows. In the next section, we construct coefficients $\{\beta_{\ell}\}_{\ell=0}^L$ explicitly. In Section \ref{barycentricform}, we present Tikhonov regularized barycentric interpolation formula and Tikhonov regularized modified Lagrange interpolation formula, which are derived from the explicit approximation polynomial \eref{solutionp} under interpolatory conditions. In Section \ref{l2quality}, we study the quality of the approximation $p_{L,N+1}\approx f$ in terms of the $L_2$ norm and the uniform norm. We give several numerical examples in Section \ref{numericalexperiments} and conclude with some remarks in Section \ref{concludingremarks}.

\section{Explicit coefficients in the Tikhonov regularized orthogonal polynomial expansion}\label{sec:explicit}
We construct coefficients $\{\beta_{\ell}\}_{\ell=0}^L$ in this section. The Tikhonov regularized approximation problem \eref{problem:l2regu} can be transformed into a matrix-form problem, which makes it easy for us to construct our desired coefficients.
\subsection{Preliminaries on Gauss quadrature weights}
Gauss quadrature occurs in almost all textbooks of numerical analysis and of orthogonal polynomials as well, and we refer to  \cite{gautschi2004orthogonal,gautschi2011numerical,kress1998numerical,von1939orthogonal}.
\begin{definition}\label{def:gauss}
Given a nonnegative weight function $w(x)$ which satisfies $\int_{-1}^1w(x){\rm d}x<\infty$, a quadrature formula
\begin{equation*}
\int_{-1}^1w(x)f(x){\rm d}x\approx\sum\limits_{j=0}^N\omega_jf(x_j)
\end{equation*}
with $N+1$ distinct quadrature points $x_0,x_1,\ldots,x_N$ is called a \emph{Gauss quadrature formula} if it integrates all polynomials $p\in\mathbb{P}_{2N+1}$ exactly, i.e., if
\begin{equation}\label{gauss}
\sum\limits_{j=0}^N\omega_jp(x_j)=\int_{-1}^1w(x)p(x){\rm d}x\quad\forall p\in\mathbb{P}_{2N+1}.
\end{equation}
Points $x_0,x_1,\ldots,x_N$ are called \emph{Gauss quadrature points}.
\end{definition}
\noindent It is well known that $N+1$ Gauss quadrature points are zeros of the orthogonal polynomial $\Phi_{N+1}$ of degree $N+1$.

\subsection{Construction of explicit coefficient}
The function $f$ sampled on $\mathcal{X}_{N+1}$ generates
\begin{equation*}
\mathbf{f}:=\mathbf{f}(\mathcal{X}_{N+1})=[f(x_0),f(x_1),\ldots,f(x_N)]^{\rm{T}}\in\mathbb{R}^{N+1},
\end{equation*}
and all Gauss quadrature weights $\omega_0,\omega_1,\ldots,\omega_N$ corresponding to $\mathcal{X}_{N+1}$ form a vector
\begin{equation*}
\mathbf{w}:=\mathbf{w}(\mathcal{X}_{N+1})=[\omega_0,\omega_1,\ldots,\omega_N]^{\rm{T}}\in\mathbb{R}^{N+1}.
\end{equation*}
Let ${\mathbf{A}}:={\mathbf{A}}(\mathcal{X}_{N+1})\in\mathbb{R}^{(N+1)\times(L+1)}$ be a matrix of orthogonal polynomials evaluated at $\mathcal{X}_{N+1}$, with entries
\begin{equation*}
{\mathbf{A}}_{j\ell}=\tilde{\Phi}_{\ell}(x_j),\quad j=0,1,\ldots,N,\quad\ell=0,1,\ldots,L.
\end{equation*}
By subtracting the structure \eref{solutionp} of approximation polynomial into the Tikhonov regularized approximation problem \eref{problem:l2regu}, the problem transforms into the following problem
\begin{equation}\label{problem:l2regumatrix}
\underset{{\bbeta}\in \mathbb{R}^{L+1}}{\min}~~\|{\bf{W}}^{\frac12}(\mathbf{A}{\bbeta}-{\rm \mathbf{f}})\|^2_2+\lambda\|{\bbeta}\|_2^2,\quad\lambda>0,
\end{equation}
where
\begin{equation*}
{\bf{W}}={\rm diag}(\omega_0,\omega_1,\ldots,\omega_N)\in\mathbb{R}^{(N+1)\times(N+1)}.
\end{equation*}

Taking the first derivative of the objective function in problem \eref{problem:l2regumatrix} with respect to $\bbeta$ leads to the first order condition
\begin{equation}\label{wproblem:l22w}
\left(\mathbf{A}^{\rm{T}}{\bf{W}}\mathbf{A}+\lambda\mathbf{I}\right){\bbeta}=\mathbf{A}^{\rm{T}}{\bf{W}}{\rm \bf{f}},\quad\lambda>0,
\end{equation}
where $\mathbf{I}\in\mathbb{R}^{(L+1)\times(L+1)}$ is an identity matrix.
It is natural to solve a system of $L+1$ linear equations using methods of numerical linear algebra, especially when $L$ is large. However, the following lemma guarantees a diagonal structure of $\mathbf{A}^{\rm{T}}{\bf{W}}\mathbf{A}+\lambda\mathbf{I}$, which is $(1+\lambda)\mathbf{I}$, thus the solution to the first order condition \eref{wproblem:l22w} can be obtained in an entry-wise closed form. When $L$ becomes large, \eref{wproblem:l22w} can still be solved fast and stably, as it is actually a scalar-vector multiplication. 
\begin{lemma}\label{thm:diagonal}
Let $\{\tilde{\Phi}_{\ell}\}_{\ell=0}^L$ be a class of orthonormal polynomials with the weight function $w(x)$, and $\mathcal{X}_{N+1}=\{x_0,x_1,\ldots,x_N\}$ be the set of zeros of $\tilde{\Phi}_{N+1}$. Assume $2L\leq2N+1$ and ${\rm \bf{w}}$ is a vector of weights satisfying the Gauss quadrature formula \eref{gauss}. Then
\begin{equation*}
\mathbf{A}^{\rm{T}}{\bf{W}}\mathbf{A}=\mathbf{I}\in\mathbb{R}^{(L+1)\times(L+1)}.
\end{equation*}
\end{lemma}
\begin{proof}
By the structure of $\mathbf{A}^{\rm{T}}{\bf{W}}\mathbf{A}$ and the exactness property \eref{gauss} of Gauss quadrature formula, we obtain
\begin{eqnarray*}
\left[\mathbf{A}^{\rm{T}}{\bf{W}}\mathbf{A}\right]_{\ell\ell'} = \sum_{j=0}^N\omega_j\tilde{\Phi}_{\ell}(x_j)\tilde{\Phi}_{\ell'}(x_j)= \int_{-1}^1w(x)\tilde{\Phi}_{\ell}(x)\tilde{\Phi}_{\ell'}(x){\rm d}x=\delta_{\ell\ell'},
\end{eqnarray*}
where $\delta_{\ell\ell'}$ is the Kronecker delta. The middle equality holds from $\tilde{\Phi}_{\ell}(x)\tilde{\Phi}_{\ell'}(x)\in\mathbb{P}_{2L}\subset\mathbb{P}_{2N+1}$, and the last equality holds because of the orthonormality of $\{\tilde{\Phi}_{\ell}\}_{\ell=0}^L$.
\end{proof}
\begin{theorem}\label{thm:l2}
Under the condition of Lemma \ref{thm:diagonal}, the optimal solution to the matrix-form Tikhonov regularized approximation problem \eref{problem:l2regumatrix} can be expressed by
\begin{equation}\label{beta:l2}
\beta_{\ell}=\frac{1}{1+\lambda}\sum_{j=0}^N\omega_j\tilde{\Phi}_{\ell}(x_j)f(x_j),\quad \ell=0,1,\ldots,L,\quad\lambda>0.
\end{equation}
Consequently, the Tikhonov regularized approximation polynomial defined by approximation problem \eref{problem:l2regu} is
\begin{equation}\label{p:l2}
p_{L,N+1}(x) =\frac{1}{1+\lambda}\sum_{\ell=0}^L\left(\sum_{j=0}^N\omega_j\tilde{\Phi}_{\ell}(x_j)f(x_j)\right)\tilde{\Phi}_{\ell}(x).
\end{equation}
\end{theorem}
\begin{proof}
This is immediately obtained from the first order condition \eref{wproblem:l22w} of the problem \eref{problem:l2regumatrix} and Lemma \ref{thm:diagonal}.
\end{proof}
\begin{remark}
When $\lambda=0$, coefficients reduce to 
\begin{equation*}
\beta_{\ell}=\sum_{j=0}^N\omega_j\tilde{\Phi}_{\ell}(x_j)f(x_j),\quad \ell=0,1,\ldots,L,
\end{equation*}
which are coefficients of hyperinterpolation on the interval $[-1,1]$ \cite{sloan1995polynomial}. Thus \eref{p:l2} could be regarded as a generalization of hyperinterpolation over the interval $[-1,1]$.
\end{remark}
In the limiting case $N\rightarrow\infty$, we have the following corollary.
\begin{corollary}\label{l2convergence}
We have the Tikhonov regularized approximation polynomial $p_{L,N+1}$ \eref{p:l2} has the uniform limit $p_{L,\infty}$ as $N\rightarrow\infty$, that is,
\begin{equation*}
\lim\limits_{N\rightarrow\infty}\|p_{L,N+1}-p_{L,\infty}\|_{\infty}=0,
\end{equation*}
where 
\begin{equation*}
p_{L,\infty}(x)=\frac{1}{1+\lambda}\sum_{\ell=0}^L\left(\int_{-1}^1w(x)\tilde{\Phi}_{\ell}(x)f(x){\rm{d}}x\right)\tilde{\Phi}_{\ell}(x).
\end{equation*}
\end{corollary}
\begin{proof}
 Let $M:=\max_{0\leq\ell\leq L}
\{\sup_{x\in[-1,1]}|\tilde{\Phi}_{\ell}(x)|\}$, then $\|p_{L,N+1}-p_{L,\infty}\|_{\infty}$ is bounded above by 
\begin{equation*}
\frac{M}{1+\lambda}\sum_{\ell=0}^L\left|\sum_{j=0}^N\omega_j\tilde{\Phi}_{\ell}(x_j)f(x_j)-\int_{-1}^1w(x)\tilde{\Phi}_{\ell}(x)f(x){\rm d}x\right|,
\end{equation*}
which converges to $0$ as $N\rightarrow\infty$ due to the convergence of Gauss quadrature formula.
\end{proof}

\section{Tikhonov Regularized barycentric interpolation formula}\label{barycentricform}
Given the explicit Tikhonov regularized approximation polynomial \eref{p:l2}, we study Tikhonov regularized approximation when $L=N$ (note that $N+1$ interpolatory points lead to an interpolant of degree $N$). We focus on barycentric interpolation formula, a fast and stable interpolation scheme, which has been made popular by Berrut and Trefethen \cite{berrut2004barycentric} in recent years. This study gives birth to Tikhonov regularized modified Lagrange interpolation and Tikhonov regularized barycentric interpolation, which will be shown to share the same computational benefits and stability properties with their classical versions, but also to have properties inherited from Tikhonov regularization. 

The barycentric interpolation is based on the Lagrange interpolation, where the interpolant is written as
\begin{equation}\label{lagrangeform}
p_N(x)=\sum\limits_{j=0}^Nf(x_j)\ell_j(x),\quad\ell_j(x)=\prod_{k\neq j}\frac{x-x_k}{x_j-x_k},\quad j=0,1,\ldots,N.
\end{equation}
An interesting rewriting of \eref{lagrangeform} is
\begin{equation}\label{equ:mdf}
p_N^{\rm{mdf}}(x)=\ell(x)\sum\limits_{j=0}^N\frac{\Omega_j}{x-x_j}f(x_j),
\end{equation}
where $\ell(x)=(x-x_0)(x-x_1)\cdots(x-x_N)$, and
\begin{equation}\label{Omega}
\Omega_j=\frac{1}{\prod_{k\neq j}(x_j-x_k)},\quad j=0,1,\ldots,N
\end{equation}
are the so-called barycentric weights. Equation \eref{equ:mdf} has been called the ``modified Lagrange formula'' by Higham \cite{higham2004numerical} and the ``first form of the barycentric interpolation formula'' by Rutishauser \cite{rutishauser1990lectures}. There is also a more elegant formula. The function values $f(x_j)\equiv1$ are obviously interpolated by $p_N^{\rm{mdf}}(x)=1$, hence \eref{equ:mdf} gives 
\begin{equation}\label{equ:lx=1}
\ell(x)\sum_{j=0}^N\frac{\Omega_j}{x-x_j}=1.
\end{equation} 
Using this equation and eliminating $\ell(x)$ in \eref{equ:mdf} gives
\begin{equation}\label{equ:bary}
p_N^{\rm{bary}}(x)=\frac{\sum\limits_{j=0}^N\Omega_jf(x_j)/(x-x_j)}{\sum\limits_{j=0}^N\Omega_j/(x-x_j)},
\end{equation}
which is called the ``second form of the barycentric interpolation formula'' by Rutishauser \cite{rutishauser1990lectures}. For details of the above derivation, we refer to the review paper by Berrut and Trefethen \cite{berrut2004barycentric}. 

The evaluation of both formulae \eref{equ:mdf} and \eref{equ:bary} is so simple. If the weights $\{\Omega_j\}$ are known or can be carried out with $\mathcal{O}(N)$ operations, both formulae produce the interpolant value evaluated at $x$ with only $\mathcal{O}(N)$ operations. Indeed, computing the weights via \eref{Omega} requires $\mathcal{O}(N^2)$ operations. However, For Chebyshev points of the first or second kind, the barycentric weights are known analytically \cite{berrut2004barycentric,salzer1972lagrangian,schwarz1989numerical}, and for other type of Jacobi points, such as Legendre points, the barycentric weights are associated with the Gauss quadrature weights, and they can be carried out with $\mathcal{O}(N)$ operations \cite{wang2014explicit,wang2012convergence} with the aid of the fast Glaser–Liu–Rokhlin algorithm \cite{glaser2007a} for Gauss quadrature. The stability properties for both formulae were also investigated by Higham \cite{higham2004numerical}. Hence barycentric interpolation formulae are fast and stable interpolation schemes. 

We call formula \eref{equ:mdf} the ``modified Lagrange interpolation formula'' and formula \eref{equ:bary} the``barycentric interpolation formula'' to distinguish them, in order to avoid the usage the ``first'' and ``second''. In the mathematical derivation, we first derive the Tikhonov regularized barycentric interpolation formula, and then derive the Tikhonov regularized modified Lagrange interpolation formula, not following the chronological order of the development of both formulae.

The Tikhonov regularized approximation polynomial \eref{p:l2} when $L=N$ can be written as 
\begin{eqnarray}\label{minimizer1}
p_{N,N+1}(x)&=\sum\limits_{\ell=0}^N\frac{\sum_{j=0}^N\omega_j\tilde{\Phi}_{\ell}(x_j)f(x_j)}{1+\lambda}\tilde{\Phi}_{\ell}(x) \nonumber\\
&=\sum_{j=0}^N\omega_jf(x_j)\sum_{\ell=0}^N\frac{\tilde{\Phi}_{\ell}(x_j)\tilde{\Phi}_{\ell}(x)}{1+\lambda}.
\end{eqnarray}
From the orthonormality of $\{\tilde{\Phi}_{\ell}(x)\}_{\ell=0}^N$ we have
\begin{eqnarray*}
\sum\limits_{j=0}^N\omega_j\sum_{\ell=0}^N\tilde{\Phi}_{\ell}(x_j)\tilde{\Phi}_{\ell}(x)
&=\sum_{\ell=0}^N\left(\sum_{j=0}^N\omega_j\tilde{\Phi}_{\ell}(x_j)\cdot1\right)\tilde{\Phi}_{\ell}(x)\\
&=\sum_{\ell=0}^N\delta_{0\ell}\|\tilde{\Phi}_0(x)\|_{L_2}\tilde{\Phi}_{\ell}(x)=         \|\tilde{\Phi}_0(x)\|_{L_2}\tilde{\Phi}_{0}(x)=1.
\end{eqnarray*}
The last equality is due to $\tilde{\Phi}_{0}(x)=\Phi_0(x)/\|\tilde{\Phi}_0(x)\|_{L_2}$ and $\Phi_0(x)=1$ for any Jacobi polynomial of degree $0$ \cite{gautschi2004orthogonal,von1939orthogonal}.
Then polynomial \eref{minimizer1} can be rewritten as
\begin{equation}\label{l2frac}
p_{N,N+1}(x)=\frac{\sum\limits_{j=0}^N\left(\omega_j\sum\limits_{\ell=0}^N\tilde{\Phi}_{\ell}(x_j)\tilde{\Phi}_{\ell}(x)\right)f(x_j)}{(1+\lambda)\sum\limits_{j=0}^N\omega_j\sum\limits_{\ell=0}^N\tilde{\Phi}_{\ell}(x_j)\tilde{\Phi}_{\ell}(x)}.
\end{equation}
By Christoffel-Darboux formula \cite[Section 1.3.3]{gautschi2004orthogonal},
\begin{eqnarray*}
 \sum\limits_{\ell=0}^N\tilde{\Phi}_{\ell}(x)\tilde{\Phi}_{\ell}(x_j)
&=\nonumber\frac{\|\Phi_{N+1}(x)\|_{L_2}}{\|\Phi_N(x)\|_{L_2}}\frac{\tilde{\Phi}_{N+1}(x)\tilde{\Phi}_{N}(x_j)-\tilde{\Phi}_{N+1}(x_j)
\tilde{\Phi}_{N}(x)}{x-x_j}\\
&=\frac{\|\Phi_{N+1}(x)\|_{L_2}}{\|\Phi_N(x)\|_{L_2}}\frac{\tilde{\Phi}_{N+1}(x)\tilde{\Phi}_{N}(x_j)}{x-x_j},
\end{eqnarray*}
with the fact that $\{x_j\}_{j=0}^N$ are zeros of $\Phi_{N+1}(x)$. By substituting the above equation into \eref{l2frac} and eliminating the common factor $\|\Phi_{N+1}(x)\|_{L_2}\tilde{\Phi}_{N+1}(x)/\|\Phi_N(x)\|_{L_2}$, which is not dependent on the index $j$, from both the numerator and the denominator, \eref{l2frac} transforms to 
\begin{equation*}
p_{N,N+1}(x)=\frac{\sum\limits_{j=0}^N\omega_j\tilde{\Phi}_N(x_j)f(x_j)/(x-x_j)}{(1+\lambda)\sum\limits_{j=0}^N\omega_j\tilde{\Phi}_N(x_j)/(x-x_j)}.
\end{equation*}
As a matter of fact, Wang, Huybrechs and Vandewalle revealed a relation $\Omega_j=\omega_j\tilde{\Phi}_{N}(x_j)$ between the barycentric weight $\Omega_j$ and the Gauss quadrature weight $\omega_j$ at $x_j$ \cite{wang2014explicit}, which finally leads to the following Tikhonov regularized barycentric interpolation formula.
\begin{theorem}
\emph{Tikhonov regularized barycentric interpolation formula.} The polynomial interpolant through data $\{f(x_j)\}_{j=0}^N$ at $N+1$ points $\{x_j\}_{j=0}^N$ is given by 
\begin{equation}\label{L2regubary}
p_{N}^{\rm{Tik-bary}}(x)=\frac{\sum\limits_{j=0}^N\Omega_jf(x_j)/(x-x_j)}{(1+\lambda)\sum\limits_{j=0}^N\Omega_j/(x-x_j)},
\end{equation}
where the weights $\{\Omega_j\}$ are defined by \eref{Omega}.
\end{theorem}
\begin{proof}
Given in the discussion above.
\end{proof}

Multiplying the Tikhonov regularized barycentric interpolation formula \eref{L2regubary} by equation \eref{equ:lx=1} gives the Tikhonov regularized modified Lagrange interpolation formula.
\begin{theorem}
\emph{Tikhonov regularized modified Lagrange interpolation formula.} The polynomial interpolant through data $\{f(x_j)\}_{j=0}^N$ at $N+1$ points $\{x_j\}_{j=0}^N$ is given by 
\begin{equation}\label{L2mdf}
p_N^{\rm{Tik-mdf}}(x)=\frac{\ell(x)}{1+\lambda}\sum\limits_{j=0}^N\frac{\Omega_j}{x-x_j}f(x_j),
\end{equation}
where the weights $\{\Omega_j\}$ are defined by \eref{Omega}.
\end{theorem}
\begin{proof}
Given in the described multiplication above the theorem. 
\end{proof}

That's it! The Tikhonov regularization only brings a multiplicative correction $1/(1+\lambda)$ into both modified Lagrange interpolation formula and barycentric interpolation formula, hence the computational benefits and stability properties for the classical version of both formulae are kept in the Tikhonov regularized version, the properties of Tikhonov regularization are also conferred to both regularized formulae. If $\lambda=0$, formulae \eref{L2mdf} and \eref{L2regubary} reduce to classical modified Lagrange interpolation formula \eref{equ:mdf} and classical barycentric interpolation formula \eref{equ:bary}, respectively.

\section{Approximation quality}\label{l2quality}
We then study the quality of the Tikhonov regularized approximation in terms of two kinds of norms and in the presence of noise. We denote by $f^{\epsilon}$ a noisy $f$, and regard both $f$ and $f^{\epsilon}$ as continuous for the following analysis. Regarding the noisy version $f^{\epsilon}$ as continuous is convenient for theoretical analysis, and is always adopted by other scholars in the field of approximation, see, for example, \cite{pereverzyev2015parameter}. We adopt this trick, and investigate the approximation properties in the sense of uniform error and $L_2$ error, respectively, that is, the uniform norm $\|f\|_{\infty}=\max_{x\in[-1,1]}|f(x)|$ and the $L_2$ norm \eref{equ:l2norm} are involved. The error of best approximation of $f$ by an element $p$ of $\mathbb{P}_L$ is also involved, which is defined by
\begin{equation*}
E_L(f):=\inf_{p\in\mathbb{P}_L}\|f-p\|_{\infty},\quad f\in \mathcal{C}([-1,1]).
\end{equation*}
By Weierstrass approximation theorem, $E_L(f)\rightarrow0$ as $L\rightarrow\infty$. We denote by $p^*$ the best approximation polynomial of degree $L$ to $f$, i.e., $E_L(f)=\|f-p^*\|_{\infty}$.

The approximation polynomial \eref{p:l2} can be deemed as an operator $\mathcal{U}_{\lambda,L,N+1}:\mathcal{C}([-1,1])\rightarrow L_2([-1,1])$ acting on $f$, i.e.,
\begin{equation*}
p_{L,N+1}(x):=\mathcal{U}_{\lambda,L,N+1}f(x):=\sum\limits_{\ell=0}^L\beta_{\ell}\tilde{\Phi}_{\ell}(x).
\end{equation*}
We can define the $L_2$ norm of the operator
\begin{equation*}
\|\mathcal{U}_{\lambda,L,N+1}\|_{L_2}:=\underset{f\neq 0}\sup\frac{\|\mathcal{U}_{\lambda,L,N+1}f\|_{L_2}}{\|f\|_{\infty}}=\underset{f\neq 0}\sup\frac{\|p_{L,N+1}\|_{L_2}}{\|f\|_{\infty}},
\end{equation*}
and the uniform norm 
\begin{equation}\label{equ:lebesgueconstant}
\|\mathcal{U}_{\lambda,L,N+1}\|_{\infty}:=\underset{f\neq 0}\sup\frac{\|\mathcal{U}_{\lambda,L,N+1}f\|_{\infty}}{\|f\|_{\infty}}=\underset{f\neq 0}\sup\frac{\|p_{L,N+1}\|_{\infty}}{\|f\|_{\infty}}.
\end{equation}
The uniform norm is none other than the Lebesgue constant (see, for example, \cite{rivlin1980introduction}), which is a tool for quantifying the divergence or convergence of polynomial approximation. 

When $\lambda=0$, the approximation polynomial reduces to the hyperinterpolation polynomial \cite{sloan1995polynomial} on $[-1,1]$:
\begin{equation}\label{p:hyperinterpolation}
\mathcal{U}_{0,L,N+1}f(x)=\sum\limits_{\ell=0}^L\left(\sum_{j=0}^N\omega_j\tilde{\Phi}_{\ell}(x_j)f(x_j)\right)\tilde{\Phi}_{\ell}(x).
\end{equation}
Apparently, given $\|\mathcal{U}_{0,L,N+1}\|_{L_2}$ and $\|\mathcal{U}_{0,L,N+1}\|_{\infty}$, Tikhonov regularization reduces both operator norms by introducing a correction factor $1/(1+\lambda)$ as $\|\mathcal{U}_{\lambda,L,N+1}f\|=\|\mathcal{U}_{0,L,N+1}f\|/(1+\lambda)$. However, the factor cannot be used for reducing approximation error, see the following analysis. When the level of noise is relatively small, it has been studied that one can directly perform denoising tasks without regularization, see, for example, \cite{hesse2017radial,le2008localized}. What is interesting for the following analysis is that Tikhonov regularization reduces operator norms but it enlarges approximation errors, and it brings a trade-off on the errors when there exists noise.

\subsection{$L_2$ norm and $L_2$ error}
Recall that the weight function $w(x)$ satisfies $\int_{-1}^1w(x){\rm d}x<\infty$, we may just as well denote by $V$ the integral. With the aid of the exactness \eref{gauss} of Gauss quadrature, we have $V=\sum_{j=0}^N\omega_j$. As a special case on the interval $[-1,1]$, Theorem 1 in \cite{sloan1995polynomial} is stated as the following lemma.
\begin{lemma}\label{lem:L2}
Let $2L\leq2N+1$. Given $f\in\mathcal{C}([-1,1])$, and let $\mathcal{U}_{0,L,N+1}f\in\mathbb{P}_L$ be defined by \eref{p:hyperinterpolation}. Then
\begin{equation}\label{equ:stabilitynoregu}
\left\|\mathcal{U}_{0,L,N+1}f\right\|_{L_2}\leq V^{1/2}\|f\|_{\infty}.
\end{equation}
\end{lemma}

\noindent With this lemma, we show Tikhonov regularization can reduce the $L_2$ norm of operator $\mathcal{U}_{\lambda,L,N+1}$ but it enlarges the approximation error $\|\mathcal{U}_{\lambda,L,N+1}f-f\|_{L_2}$.
\begin{proposition}\label{thm:L2}
Let $2L\leq2N+1$. Given $f\in\mathcal{C}([-1,1])$, and let $\mathcal{U}_{\lambda,L,N+1}f\in\mathbb{P}_L$ be defined by \eref{p:l2}. Then
\begin{equation}\label{equ:stability}
\left\|\mathcal{U}_{\lambda,L,N+1}f\right\|_{L_2}\leq \frac{V^{1/2}}{1+\lambda}\|f\|_{\infty},
\end{equation}
and 
\begin{equation}\label{equ:L2error}
\|\mathcal{U}_{\lambda,L,N+1}f-f\|_{L_2}\leq\left(1+\frac{1}{1+\lambda}\right)E_L(f)+\frac{\lambda}{1+\lambda}\|p^*\|_{L_2}.
\end{equation}
Thus
\begin{equation*}\label{equ:L2errorlimits}
\lim_{L\rightarrow\infty}\|\mathcal{U}_{\lambda,L,N+1}f-f\|_{L_2}\leq\frac{\lambda}{1+\lambda}\|p^*\|_{L_2}.
\end{equation*}
\end{proposition}

\begin{proof}
The stability result \eref{equ:stability} follows from $
\left\|\mathcal{U}_{\lambda,L,N+1}f\right\|_{L_2}=\left\|\mathcal{U}_{0,L,N+1}f\right\|_{L_2}/(1+\lambda)$ and Lemma \ref{lem:L2}. Note that for all $g\in \mathcal{C}([-1,1])$, from Cauchy-Schwarz inequality there exists $\|g\|_{L_2}=\sqrt{\left<g,g\right>_{L_2}}\leq\|g\|_{\infty}\sqrt{\left<1,1\right>_{L_2}}=V^{1/2}\|g\|_{\infty}$, and also note that for all $p\in\mathbb{P}_L$, $\mathcal{U}_{\lambda,L,N+1}p\neq p$ but from \eref{p:l2} we obtain
\begin{equation*}
\mathcal{U}_{\lambda,L,N+1}p=(\mathcal{U}_{0,L,N+1}p)/(1+\lambda)=p/(1+\lambda)
\end{equation*}
as $\mathcal{U}_{0,L,N+1}p=p$ (shown in \cite[Lemma 4]{sloan1995polynomial}). Then for any polynomial $p\in\mathbb{P}_L$, 
\begin{eqnarray*}
\fl \|\mathcal{U}_{\lambda,L,N+1}f-f\|_{L_2}&=\|\mathcal{U}_{\lambda,L,N+1}(f-p)-(f-p)-(p-\mathcal{U}_{\lambda,L,N+1}p)\|_{L_2}\\
&\leq\|\mathcal{U}_{\lambda,L,N+1}(f-p)\|_{L_2}+\|f-p\|_{L_2}+\|p-\mathcal{U}_{\lambda,L,N+1}p\|_{L_2}\\
&\leq\frac{V^{1/2}}{1+\lambda}\|f-p\|_{\infty}+V^{1/2}\|f-p\|_{\infty}+\frac{\lambda}{1+\lambda}\|p\|_{L_2}.
\end{eqnarray*}
As the above inequality holds for any polynomials, letting $p$ be $p^*$ leads to \eref{equ:L2error}.
\end{proof}
Proposition \ref{thm:L2} indicates that when there is no noise, we should avoid introducing regularization; however, when data are contaminated by noise, Tikhonov regularization can reduce a new error term introduced by noise.  
\begin{theorem}\label{thm:L2noise}
Let $2L\leq2N+1$. Given $f\in\mathcal{C}([-1,1])$ and its noisy version $f^{\epsilon}\in\mathcal{C}([-1,1])$, and let $\mathcal{U}_{\lambda,L,N+1}f\in\mathbb{P}_L$ be defined by \eref{p:l2}. Then
\begin{equation}\label{equ:L2errornoise}
\fl \|\mathcal{U}_{\lambda,L,N+1}f^{\epsilon}-f\|_{L_2}\leq\frac{V^{1/2}}{1+\lambda}\|f-f^{\epsilon}\|_{\infty}+\left(1+\frac{1}{1+\lambda}\right)E_L(f)+\frac{\lambda}{1+\lambda}\|p^*\|_{L_2}.
\end{equation}
\end{theorem}
\begin{proof}
For any polynomial $p\in\mathbb{P}_L$, 
\begin{eqnarray*}
\fl \|\mathcal{U}_{\lambda,L,N+1}f^{\epsilon}-f\|_{L_2}&=\|\mathcal{U}_{\lambda,L,N+1}(f^{\epsilon}-p)-(f-p)-(p-\mathcal{U}_{\lambda,L,N+1}p)\|_{L_2}\\
&\leq\|\mathcal{U}_{\lambda,L,N+1}(f^{\epsilon}-p)\|_{L_2}+\|f-p\|_{L_2}+\|p-\mathcal{U}_{\lambda,L,N+1}p\|_{L_2}\\
&\leq\frac{V^{1/2}}{1+\lambda}\|f^{\epsilon}-p\|_{\infty}+V^{1/2}\|f-p\|_{\infty}+\frac{\lambda}{1+\lambda}\|p\|_{L_2}.
\end{eqnarray*}
Estimating $\|f^{\epsilon}-p\|_{\infty}$ by $\|f^{\epsilon}-p\|_{\infty}\leq\|f^{\epsilon}-f\|_{\infty}+\|f-p\|_{\infty}$ and letting $p$ be $p^*$ lead to \eref{equ:L2errornoise}.
\end{proof}
\begin{remark}
When there exists noise and $\lambda=0$, there holds
\begin{equation*}
\|\mathcal{U}_{0,L,N+1}f^{\epsilon}-f\|_{L_2}\leq V^{1/2}\|f-f^{\epsilon}\|_{\infty}+2E_L(f),
\end{equation*}
which enlarges the part $V^{1/2}\|f-f^{\epsilon}\|_{\infty}/({1+\lambda})+\left(1+1/(1+\lambda)\right)E_L(f)$ in \eref{equ:L2errornoise} but vanishes the part $\lambda\|p^*\|_{L_2}/(1+\lambda)$. Hence there should be a trade-off strategy for $\lambda$ in practice.
\end{remark}

\subsection{Uniform norm (Lebesgue constant) and uniform error}
The uniform case provides the similar information on the Tikhonov regularization as the $L_2$ case. Let 
\begin{equation}\label{lebesguehyperintepolation}
\Lambda_L:=\underset{f\neq 0}\sup\frac{\|\mathcal{U}_{0,L,N+1}f\|_{\infty}}{\|f\|_{\infty}}
\end{equation}
be the Lebesgue constant for hyperinterpolation $\mathcal{U}_{0,L,N+1}$ of degree $L$. It is obvious that Tikhonov regularization can reduce the Lebesgue constant \eref{lebesguehyperintepolation}.
\begin{proposition}\label{thm:reducelebesgue}
Let $\Lambda_L$ be the Lebesgue constant for hyperinterpolation $\mathcal{U}_{0,L,N+1}$ of $\mathcal{C}([-1,1])$ onto $\mathbb{P}_{L}$, and let $\Lambda_{\lambda,L}$ be the Lebesgue constant for Tikhonov regularized approximation $\mathcal{U}_{\lambda,L,N+1}$ of $\mathcal{C}([-1,1])$ onto $\mathbb{P}_{L}$. Then
\begin{equation*}
\Lambda_{\lambda,L}:=\|\mathcal{U}_{\lambda,L,N+1}\|_{\infty}=\frac{1}{1+\lambda}\Lambda_L.
\end{equation*}
\end{proposition}
\begin{proof}
For any $f\in\mathcal{C}([-1,1])$, there holds $\mathcal{U}_{0,L,N+1}f=\mathcal{U}_{\lambda,L,N+1}f/(1+\lambda)$, then
\begin{equation*}
\Lambda_{\lambda,L}=\underset{f\neq 0}\sup\frac{\|\mathcal{U}_{\lambda,L,N+1}f\|_{\infty}}{\|f\|_{\infty}}=\frac{1}{1+\lambda}
\underset{f\neq 0}\sup\frac{\|\mathcal{U}_{0,L,N+1}f\|_{\infty}}{\|f\|_{\infty}}=\frac{1}{1+\lambda}\Lambda_L.
\end{equation*} 
Hence we prove this proposition.
\end{proof}
\begin{remark}
As hyperinterpolation reduces to interpolation when $L=N$ \cite{sloan1995polynomial}, Tikhonov regularization can also reduce Lebesgue constants of classical interpolation.
\end{remark}

Though Lebesgue constants are reduced by introducing regularization, approximation errors may be enlarged.
\begin{proposition}
Let $2L\leq2N+1$. Given $f\in\mathcal{C}([-1,1])$, and let $\mathcal{U}_{\lambda,L,N+1}f\in\mathbb{P}_L$ be defined by \eref{p:l2}. Then
\begin{equation*}
\|\mathcal{U}_{\lambda,L,N+1}f-f\|_{\infty}\leq(1+\Lambda_{\lambda,L})E_L(f)+\frac{\lambda}{1+\lambda}\|p^*\|_{\infty}.
\end{equation*}
\end{proposition}
\begin{proof}
By the definition \eref{equ:lebesgueconstant} of Lebesgue constant of Tikhonov regularized approximation, $\|\mathcal{U}_{\lambda,L,N+1}(f-p^*)\|_{\infty}$ is not greater than $\Lambda_{\lambda,L}\|f-p^*\|_{\infty}$, thus
\begin{eqnarray}
 \|\mathcal{U}_{\lambda,L,N+1}f-p^*\|_{\infty}&\leq\Lambda_{\lambda,L}\|f-p^*\|_{\infty}+\|p^*-\mathcal{U}_{\lambda,L,N+1}p^*\|_{\infty}\nonumber\\
&=\Lambda_{\lambda,L}\|f-p^*\|_{\infty}+\frac{\lambda}{1+\lambda}\|p^*\|_{\infty}\label{equ:replacing}
\end{eqnarray}
as $\mathcal{U}_{\lambda,L,N+1}(f-p^*)=(\mathcal{U}_{\lambda,L,N+1}f-p^*)+(p^*-\mathcal{U}_{\lambda,L,N+1}p^*)$. Then the decomposition $\mathcal{U}_{\lambda,L,N+1}f-f=(\mathcal{U}_{\lambda,L,N+1}f-p^*)-(f-p^*)$ completes the proof.
\end{proof}
\begin{remark}
Comparing with the classical near-best approximation property $\|\mathcal{U}_{0,L,N+1}f-f\|_{\infty}\leq(1+\Lambda_{L})E_L(f)$, Tikhonov regularization reduces the part $(1+\Lambda_{L})E_L(f)$ but introduces a new part $\lambda\|p^*\|_{\infty}/(1+\lambda)$.
\end{remark}

\begin{theorem}
Let $2L\leq2N+1$. Given $f\in\mathcal{C}([-1,1])$ and its noisy version $f^{\epsilon}\in\mathcal{C}([-1,1])$, and let $\mathcal{U}_{\lambda,L,N+1}f\in\mathbb{P}_L$ be defined by \eref{p:l2}. Then
\begin{equation*}
\|\mathcal{U}_{\lambda,L,N+1}f^{\epsilon}-f\|_{\infty}\leq\Lambda_{\lambda,L}\|f^{\epsilon}-f\|_{\infty}+(1+\Lambda_{\lambda,L})E_{L}(f)+\frac{\lambda}{1+\lambda}\|p^*\|_{\infty}.
\end{equation*}
\end{theorem}
\begin{proof}
Since $\mathcal{U}_{\lambda,L,N+1}f^{\epsilon}-f=(\mathcal{U}_{\lambda,L,N+1}f^{\epsilon}-p^*)-(f-p^*)$, replacing $f$ by $f^{\epsilon}$ in \eref{equ:replacing} leads to
\begin{equation*}
\|\mathcal{U}_{\lambda,L,N+1}f^{\epsilon}-f\|_{\infty}=\Lambda_{\lambda,L}\|f^{\epsilon}-p^*\|_{\infty}+\frac{\lambda}{1+\lambda}\|p^*\|_{\infty}+\|f-p^*\|_{\infty}.
\end{equation*}
The decomposition $\|f^{\epsilon}-p^*\|_{\infty}\leq\|f^{\epsilon}-f\|_{\infty}+\|f-p^*\|_{\infty}$ completes the proof of the theorem. 
\end{proof}
\begin{remark}
When there exists noise and $\lambda=0$, there holds
\begin{equation*}
\|\mathcal{U}_{0,L,N+1}f^{\epsilon}-f\|_{\infty}\leq\Lambda_{L}\|f^{\epsilon}-f\|_{\infty}+(1+\Lambda_{L})E_{L}(f).
\end{equation*}
Recall that $\Lambda_{\lambda,L}<\Lambda_L$ if $\lambda>0$. The theorem asserts that Tikhonov regularization can reduce the error introduced by noise, and indicates again that there should be a trade-off strategy for $\lambda$ in practice.
\end{remark}

\section{Numerical experiments}\label{numericalexperiments}
In this section, we report numerical results to illustrate the theoretical results derived above and test the efficiency of the Tikhonov regularized approximation in Gauss quadrature points. Three testing functions are involved in the following experiments, which are a function given in \cite{berrut2004barycentric} 
\begin{equation*}
\displaystyle
f_1(x)=|x|+\frac{x}{2}-x^2,
\end{equation*}
an Airy function
\begin{equation*}
f_2(x)={\rm{Airy}}(40x),
\end{equation*}
and a rather wiggly function given in \cite{trefethen2013approximation}
\begin{equation*}
f_3(x)=\tanh(20\sin(12x))+0.02{\rm{e}}^{3x}\sin(300x).
\end{equation*}
Commands for computing Gauss quadrature points and weights, and barycentric weights are included in \textsc{Chebfun} 5.7.0 \cite{driscoll2014chebfun}. All numerical results are carried out by using MATLAB R2020a on a laptop (16 GB RAM, Intel® CoreTM i7-9750H Processor) with macOS Catalina.

We adopt the uniform error and the $L_2$ error to test the efficiency of approximation, which are estimated as follows. The uniform error of the approximation is estimated by
  \begin{eqnarray*}
  \|f(x)-p_{L,N+1}(x)\|_{\infty}&:=\underset{x\in[-1,1]}{\max}|f(x)-p_{L,N+1}(x)|\\
                           &\simeq\underset{x\in\mathcal{X}}{\max}|f(x)-p_{L,N+1}(x)|,
  \end{eqnarray*}
  where $\mathcal{X}$ is a large but finite set of well distributed points over the interval $[-1,1]$. The $L_2$ error of the approximation is estimated by a proper Gauss quadrature rule:
  \begin{eqnarray*}
  \|f(x)-p_{L,N+1}(x)\|_{L_2} &=\left(\int_{-1}^1w(x)(f(x)-p_{L,N+1}(x))^2{\rm d}x\right)^{1/2}\\
                        &\simeq\left(\sum_{j=0}^N\omega_j(f(x_j)-p_{L,N+1}(x_j))^2\right)^{1/2}.
  \end{eqnarray*}

We first test the efficiency of approximation scheme \eref{p:l2} of $f_1(x)$ and $f_2(x)$ by normalized Chebyshev polynomials of the first kind with data sampled on Gauss-Chebyshev points of the first kind in the presence of noise. The level of noise is measured by \emph{signal-to-noise ratio} (SNR), which is defined as the ratio of signal power to the noise power, and is often expressed in decibels (dB). A lower scale of SNR suggests more noisy data. We take $\lambda=10^{-2},10^{-1.9},\ldots,10^{-0.1},1$ to choose the best regularization parameter. Here we choose $\lambda=10^{-0.7}$ for all the following experiments. For more advanced and adaptive methods to choose the parameter $\lambda$, we refer to \cite{lazarov2007balancing,lu2010discrepancy,pereverzyev2015parameter}. Fix $N=500$, let $L$ be increasing from $10$ to $N$, and add 5dB Gauss white noise onto sampled data. Uniform errors and $L_2$ errors for approximations of both $f_1(x)$ and $f_2(x)$ are shown in Figure \ref{figure1}, illustrating that the Tikhonov regularization can reduce noise, especially when $L$ becomes large. The enlarging gap between $L_2$ errors is due to a fact that increasing $L$ requires more data but the data size is fixed (fixed $N$), hence the gap also suggests that Tikhonov regularization can handle this data shortage issue.  
\begin{figure}[htbp]
  \centering
  \includegraphics[width=\textwidth]{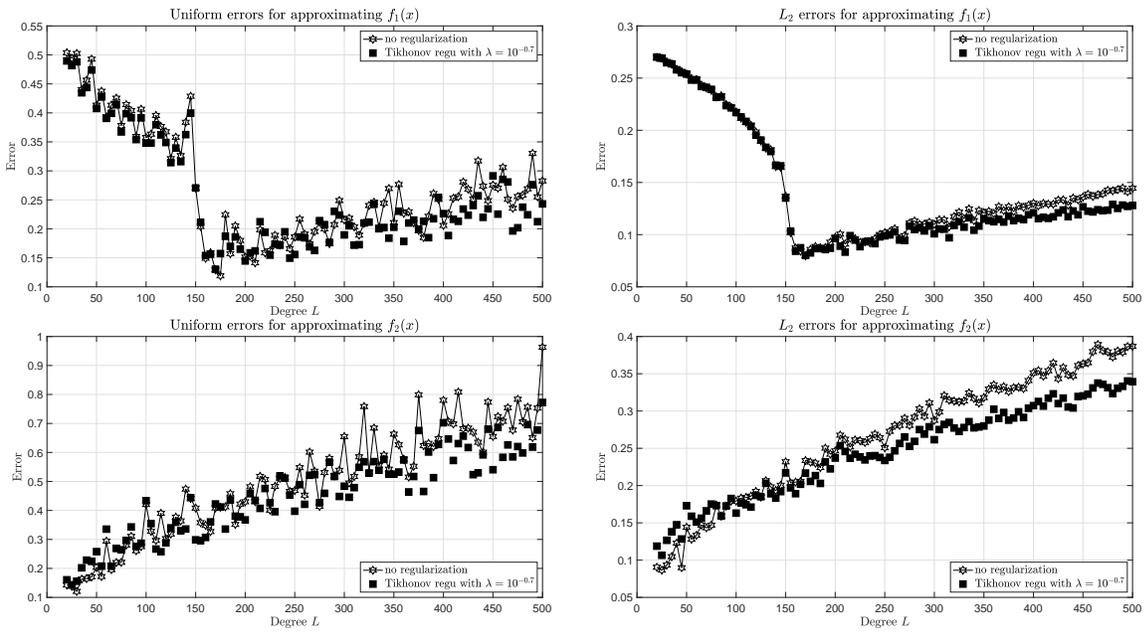}\\
  \caption{Computational results on approximation scheme \eref{p:l2} with fixed $N=500$ and increasing $L$ from 10 to $N$, in the presence of 5dB Gauss white noise.}\label{figure1}
\end{figure}

On the other hand, if we fix $L=500$ and let $N$ be increasing from $500$ to $2000$, that is, data size is increasing, then Figure \ref{figure2} describes decreasing uniform errors and $L_2$ errors with respect to $N$. The starting value of $N$ is 500 since Gauss quadrature would lose its exactness if $N\leq L$. Computational results plotted in Figure \ref{figure2} also assert that the Tikhonov regularization can reduce noise, especially when $N$ is small. In this case, the gap becomes narrow as $N$ increasing, which is due to the same fact that more data lead to better performance. This narrowing gap also indicates that Tikhonov regularization can handle this data shortage issue.
\begin{figure}[htbp]
  \centering
  \includegraphics[width=\textwidth]{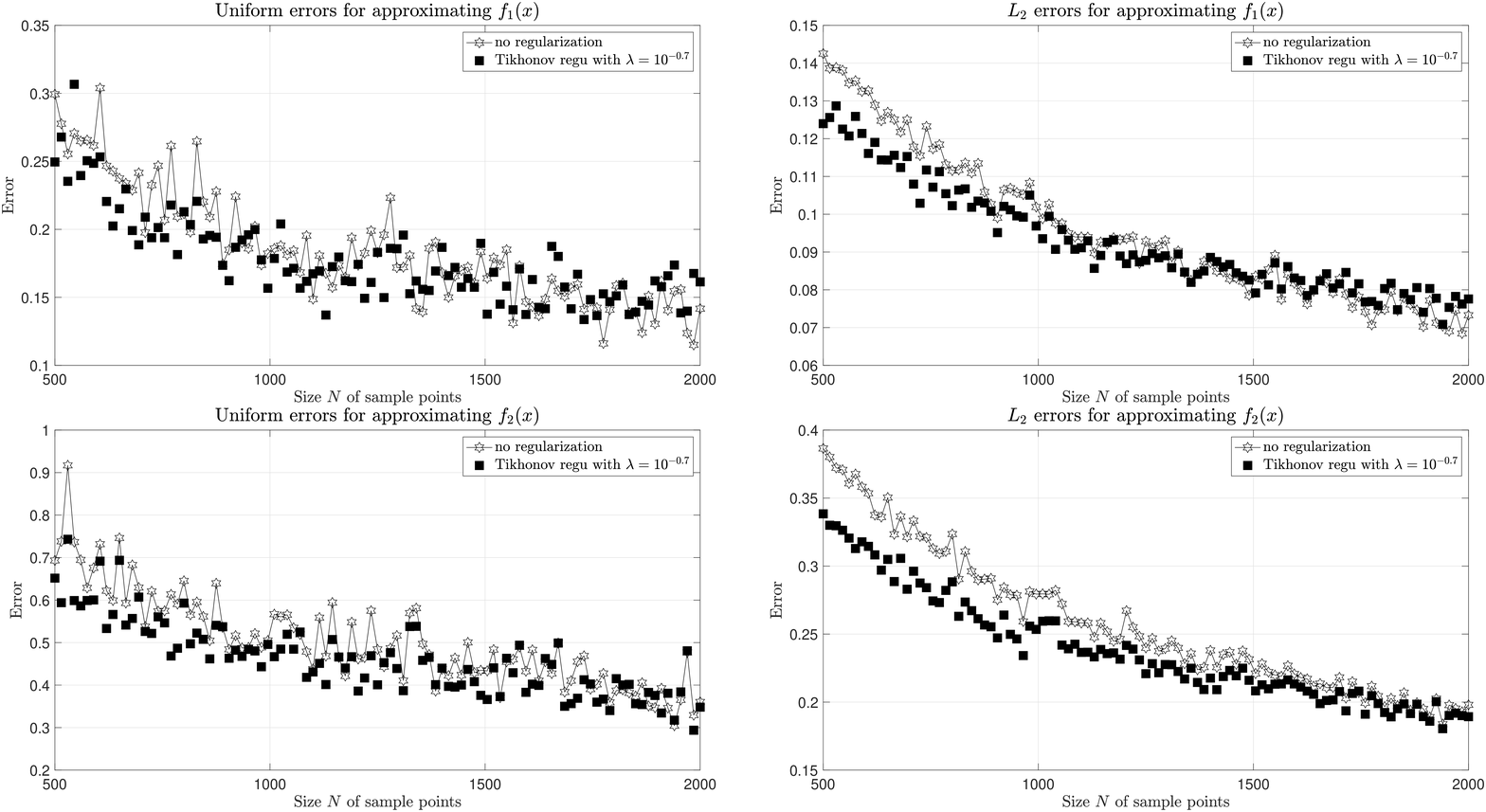}\\
  \caption{Computational results on approximation scheme \eref{p:l2} with fixed $L=500$ and increasing $N$ from 500 to 2000, in the presence of 5dB Gauss white noise.}\label{figure2}
\end{figure}

We then test the efficiency of Tikhonov regularized barycentric interpolation formula \eref{L2regubary} in approximating $f_3(x)$, with data sampled on Gauss-Chebyshev points of the first kind. The experiment is conducted via the barycentric interpolation scheme \eref{L2regubary} rather than the approximation scheme \eref{p:l2} under interpolatory conditions. Computational results in Figure \ref{figure3} show that Tikhonov regularized barycentric interpolation works better than classical barycentric interpolation in the presence of noise. However, in the noise-free case, both kinds of errors for classical barycentric interpolation decline to 0 as $L$ increasing but those for Tikhonov regularized case do not. This misconvergence result of Tikhonov regularized barycentric interpolation, in another perspective, is a good agreement with the theoretical result that regularization would introduce an additional error $\lambda\|p^*\|_{L_2}/(1+\lambda)$ into the $L_2$ error bound \eref{equ:L2error}, and this error is around $0.3$ in this experiment. 
\begin{figure}[htb]
  \centering
  \includegraphics[width=\textwidth]{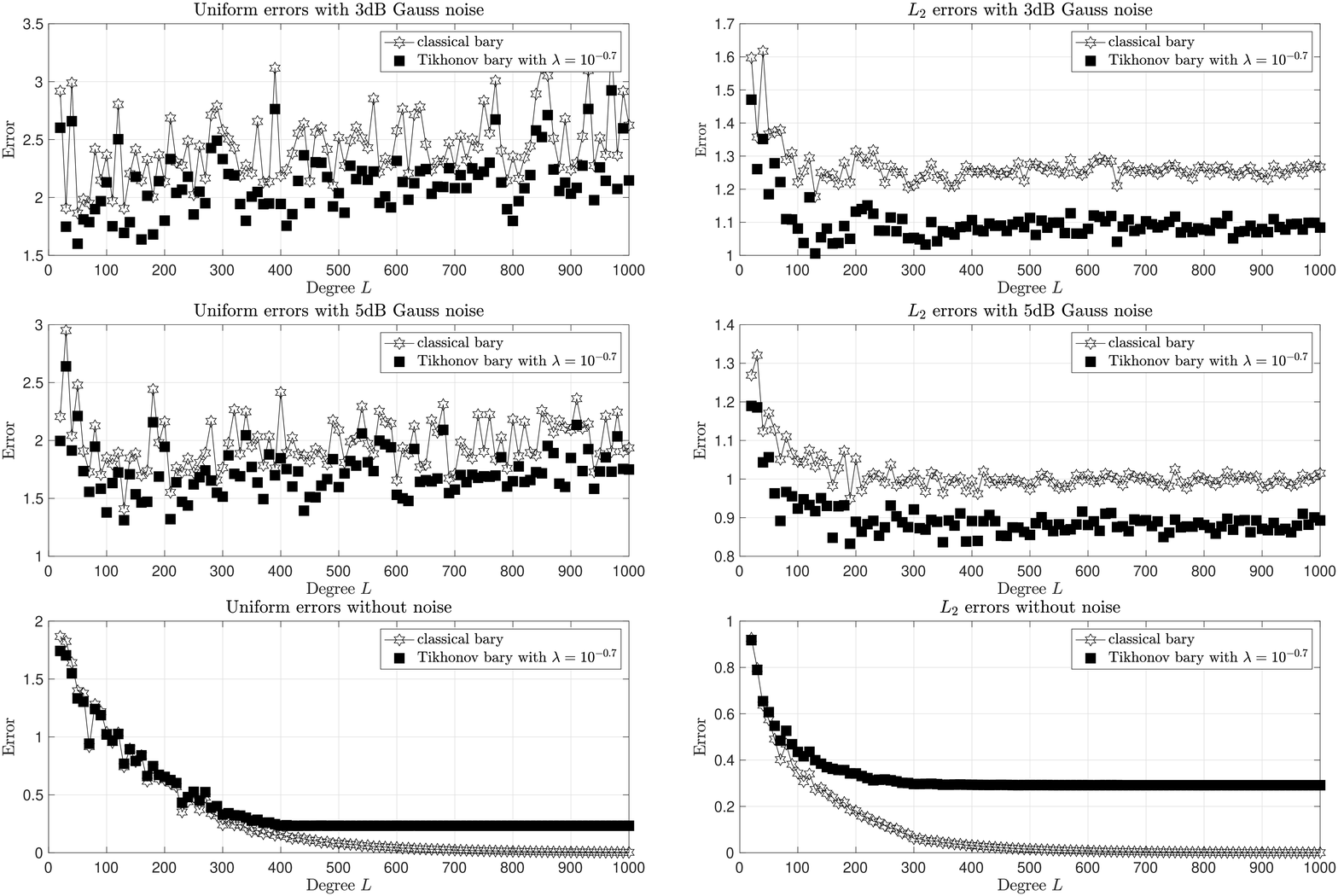}\\
  \caption{Computational results of classical barycentric formula \eref{equ:bary} and Tikhonov regularized barycentric formula \eref{L2regubary} in approximating $f_3(x)$, with the number $N$ of interpolatory points increasing from 20 to 1000.}\label{figure3}
\end{figure}

At last, we take a certain $N$, say $N=60$, and test on function $f_1(x)$. Figure \ref{figure4} reports the results, and ``true data'' in all subfigures denotes values of $f_1(x)$ at $61$ Gauss-Chebyshev points of the first kind. When data is sampled via $f_1(x)$, that is, there is no noise in sampling, as shown in the above experiment, regularization is not needed. For numerical results on classical barycentric interpolation with true data, we refer to \cite{berrut2004barycentric,wang2014explicit}. When data is sampled via a multiple of $f_1(x)$, which is $1.2f_1(x)$ here, true data and Tikhonov regularized interpolant appear to be in a good agreement, which is due to $1.2/(1+\lambda)=1.0004\approx1$ with $\lambda=10^{-0.7}$. We then test on different levels of additive random noise. Data are sampled via $(1+0.3r)*f_1(x_j)$ and $(1+0.4r)*f_1(x_j)$, respectively, where $j=0,1,\ldots,N$, and $r$ is a random number in $(0,1)$ which is generated by MATLAB command \texttt{rand(1)}. Tikhonov regularized barycentric formula performs better than the classical formula when the level of noise becomes large, especially near both endpoints. 
\begin{figure}[htbp]
  \centering
  \includegraphics[width=\textwidth]{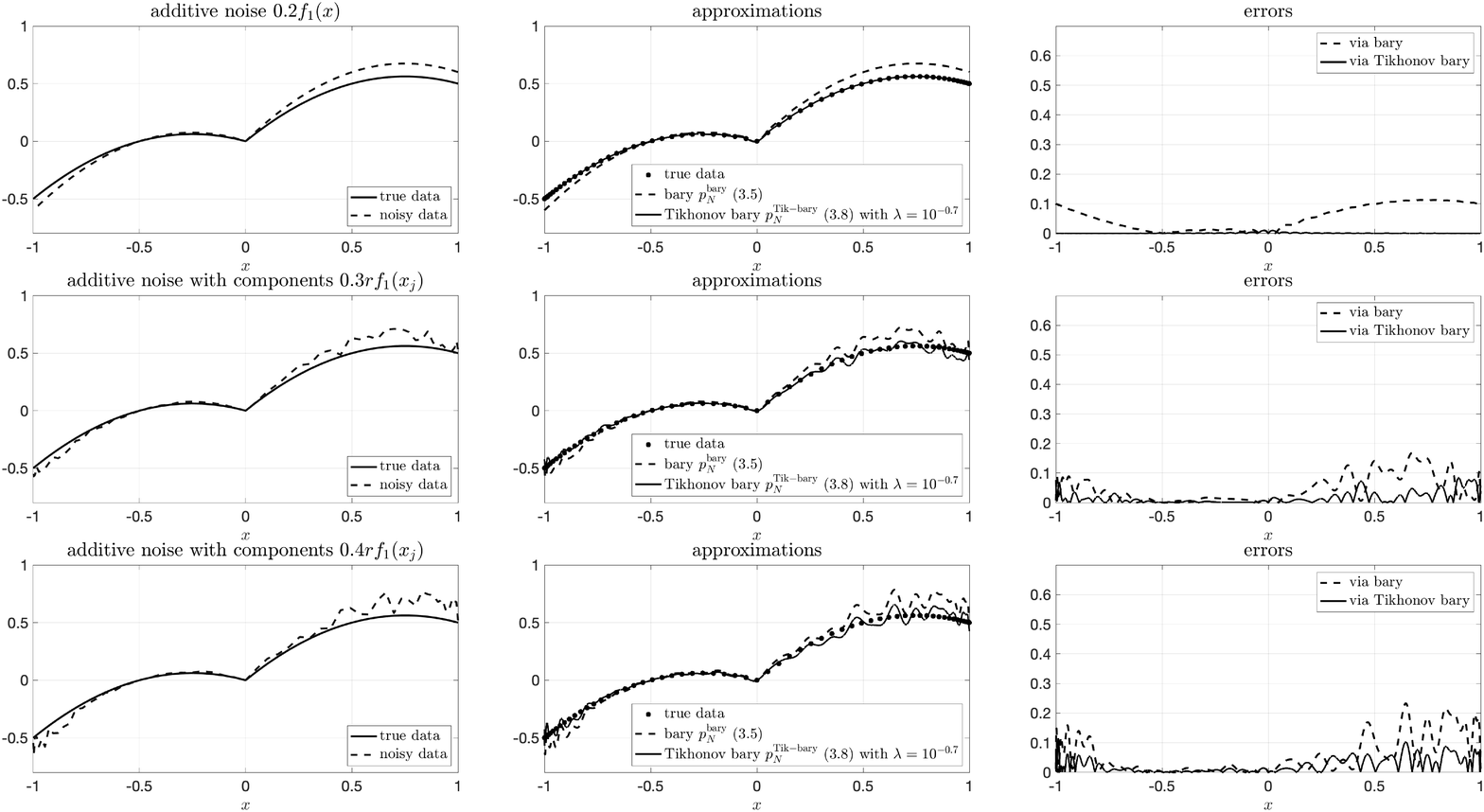}\\
  \caption{Left column: true data (solid) and noisy data (dash). Middle column: interpolants obtained by classical barycentric formula (dash) and Tikhonov regularized barycentric formula (solid) in 61 points. Right column: errors $|p_N^{\rm{bary}}(x)-f_1(x)|$ (dash) and $|p_N^{\rm{Tik-bary}}(x)-f_1(x)|$ (solid).}\label{figure4}
\end{figure}

If we add an oscillating term $\sin(10x)$ onto $f_1(x)$, plots in Figure \ref{figure5} show the similar results with those in Figure \ref{figure4}. In this figure, Tikhonov regularized barycentric formula also performs better than the classical formula in concerned levels of noise, especially near extreme points of $f_1(x)+\sin(10x)$. 
\begin{figure}[htbp]
  \centering
  \includegraphics[width=\textwidth]{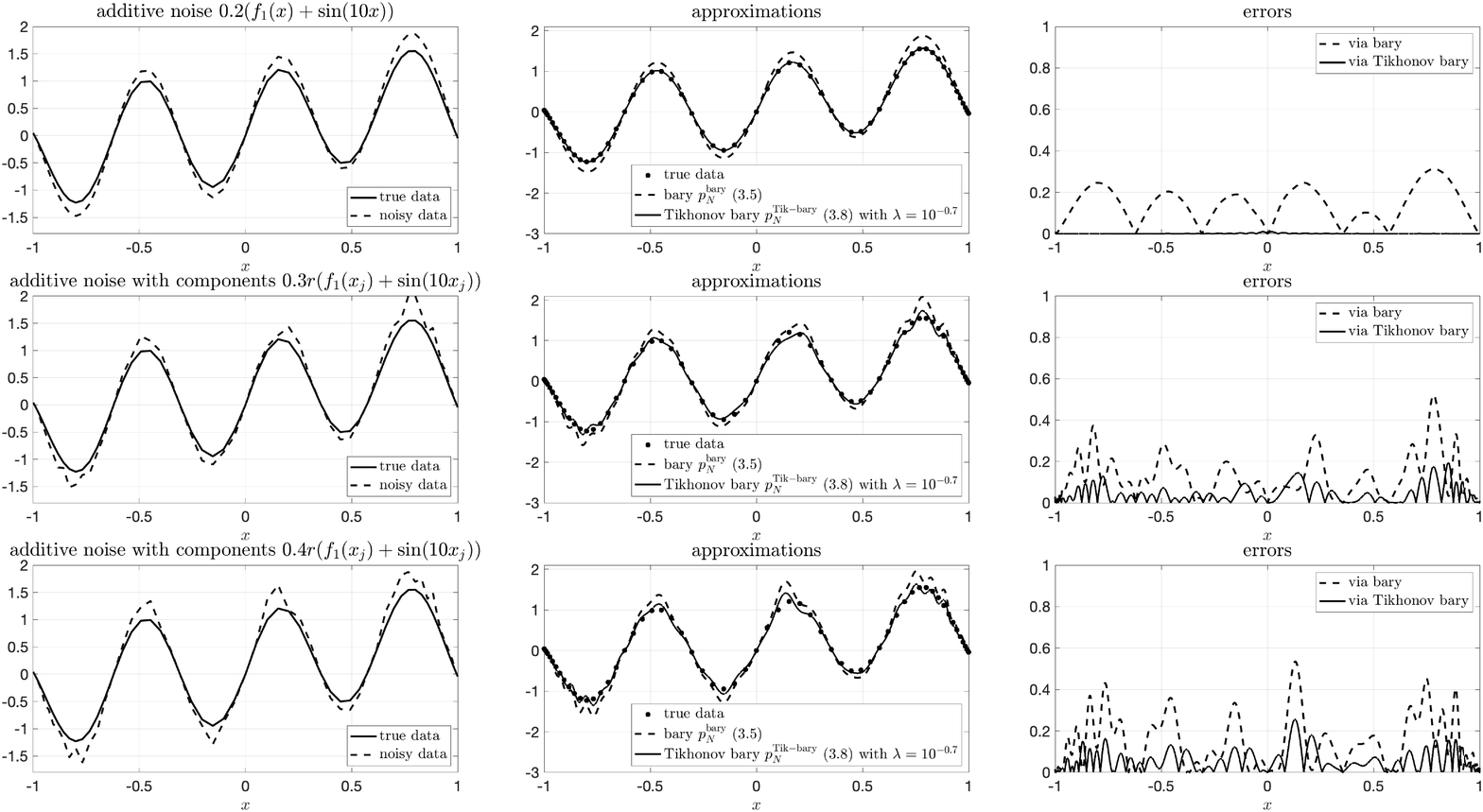}\\
  \caption{Left column: true data (solid) and noisy data (dash). Middle column: interpolants obtained by classical barycentric formula (dash) and Tikhonov regularized barycentric formula (solid) in 61 points. Right column: errors $|p_N^{\rm{bary}}(x)-(f_1(x)+\sin(10x))|$ (dash) and $|p_N^{\rm{Tik-bary}}(x)-(f_1(x)+\sin(10x))|$ (solid).}\label{figure5}
\end{figure}

\section{Concluding remarks}\label{concludingremarks}
What we have seen from the above is that Tikhonov regularization can reduce noise in sampling data with an approximation scheme, in terms of reducing Lebesgue constants and the error term relating to noise. But it also introduces an additional error term, hence a trade-off strategy should be customized in practice. These findings also suit for the newly presented Tikhonov regularized barycentric formulae. While solving this approximation problem, it is shown that proper choice of orthonormal polynomials and Gauss quadrature points leads to entry-wise closed-form solutions to the problem, which simplifies the analysis on the approximation scheme. Gauss quadrature rules can be relaxed by using the concept of Marcinkiewicz-Zygmund quadrature measure (see, for example, \cite{filbir2011marcinkiewicz,mhaskar2004polynomial,mhaskar2005polynomial,mhaskar2001spherical}), thus results in this paper can be generalized based on Marcinkiewicz-Zygmund quadrature. The reason why Gauss quadrature rules are adopted in this paper is two-fold: on the one hand, using Gauss quadrature is enough to provide a window into the behavior of approximation error in the usage of Tikhonov regularization and in the presence of noise; on the other hand, the development of barycentric interpolation formulae is based on Gauss quadrature points \cite{berrut2004barycentric}.

Although we only consider the simplest Tihonov regularization term, it also provides some useful information that regularization may improve performance of polynomial approximation. In inverse problems, statistics, and machine learning, different kinds of regularization terms are developed. We may consider other regularization techniques and derive other regularized barycentric interpolation formulae in the future. With the fast and stable property of barycentric formulae, regularized barycentric formulae provide a flexible choice for polynomial interpolation in noisy case, which only introduces a multiplicative factor $1/(1+\lambda)$ or maybe other corrective factors derived in the future. Last but far from the least, the choices of the polynomial degree $L$ and regularization parameter $\lambda$ deserve future studies. The polynomial degree has already been considered as a regularization parameter to deal with ill-conditioned issues \cite{lu2013legendre,mhaskar2020direct,mhaskar2013filtered,mhaskar2017deep}, and we expect an appropriate choice of $L$ to make better performance in the context of Tikhonov regularized polynomial approximation. Regularization parameter $\lambda$, which is chosen in a manual way in numerical experiments, also needs to be addressed more adaptively, see, for example, \cite{lazarov2007balancing,lu2010discrepancy,pereverzyev2015parameter}.

\section*{Acknowledgment}
The authors are greatly thankful to the anonymous referees for valuable suggestions that helped to improve this paper.

\section*{References}
\bibliographystyle{acm}
\bibliography{IP-refs}

\begin{thebibliography}{10}

\bibitem{berrut2004barycentric}
{\sc Berrut, J.-P., and Trefethen, L.~N.}
\newblock Barycentric {L}agrange interpolation.
\newblock {\em SIAM Review 46}, 3 (2004), 501--517.

\bibitem{driscoll2014chebfun}
{\sc Driscoll, T.~A., Hale, N., and Trefethen, L.~N.}
\newblock {\em Chebfun Guide}.
\newblock Pafnuty Publications, Oxford, 2014.

\bibitem{filbir2011marcinkiewicz}
{\sc Filbir, F., and Mhaskar, H.~N.}
\newblock Marcinkiewicz--{Z}ygmund measures on manifolds.
\newblock {\em Journal of Complexity 27}, 6 (2011), 568--596.

\bibitem{gautschi2004orthogonal}
{\sc Gautschi, W.}
\newblock {\em Orthogonal Polynomials: Computation and Approximation}.
\newblock Oxford University Press, Oxford, 2004.

\bibitem{gautschi2011numerical}
{\sc Gautschi, W.}
\newblock {\em Numerical analysis 2nd edition}.
\newblock Birkh\"{a}user, Basel, 2012.

\bibitem{glaser2007a}
{\sc Glaser, A., Liu, X., and Rokhlin, V.}
\newblock A fast algorithm for the calculation of the roots of special
  functions.
\newblock {\em SIAM Journal on Scientific Computing 29}, 4 (2007), 1420--1438.

\bibitem{hesse2017radial}
{\sc Hesse, K., Sloan, I.~H., and Womersley, R.~S.}
\newblock Radial basis function approximation of noisy scattered data on the
  sphere.
\newblock {\em Numerische Mathematik 137}, 3 (2017), 579--605.

\bibitem{higham2004numerical}
{\sc Higham, N.~J.}
\newblock The numerical stability of barycentric {L}agrange interpolation.
\newblock {\em IMA Journal of Numerical Analysis 24}, 4 (2004), 547--556.

\bibitem{kress1998numerical}
{\sc Kress, R.}
\newblock {\em Numerical analysis}, vol.~181 of {\em Graduate Texts in
  Mathematics}.
\newblock Springer, New York, 1998.

\bibitem{lazarov2007balancing}
{\sc Lazarov, R.~D., Lu, S., and Pereverzev, S.~V.}
\newblock On the balancing principle for some problems of numerical analysis.
\newblock {\em Numerische Mathematik 106}, 4 (2007), 659--689.

\bibitem{le2008localized}
{\sc Le~Gia, Q.~T., and Mhaskar, H.~N.}
\newblock Localized linear polynomial operators and quadrature formulas on the
  sphere.
\newblock {\em SIAM Journal on Numerical Analysis\/} (2008), 440--466.

\bibitem{lu2013legendre}
{\sc Lu, S., Naumova, V., and Pereverzev, S.~V.}
\newblock Legendre polynomials as a recommended basis for numerical
  differentiation in the presence of stochastic white noise.
\newblock {\em Journal of Inverse and Ill-posed Problems 21}, 2 (2013),
  193--216.

\bibitem{lu2009sparse}
{\sc Lu, S., and Pereverzev, S.~V.}
\newblock Sparse recovery by the standard {T}ikhonov method.
\newblock {\em Numerische Mathematik 112}, 3 (2009), 403--424.

\bibitem{lu2013regularization}
{\sc Lu, S., and Pereverzev, S.~V.}
\newblock {\em Regularization theory for ill-posed problems: {S}elected
  topics}, vol.~58.
\newblock De Gruyter, Berlin, 2013.

\bibitem{lu2006analysis}
{\sc Lu, S., Pereverzev, S.~V., and Ramlau, R.}
\newblock An analysis of {T}ikhonov regularization for nonlinear ill-posed
  problems under a general smoothness assumption.
\newblock {\em Inverse Problems 23}, 1 (2006), 217.

\bibitem{lu2010discrepancy}
{\sc Lu, S., Pereverzev, S.~V., Shao, Y., and Tautenhahn, U.}
\newblock Discrepancy curves for multi-parameter regularization.
\newblock {\em Journal of Inverse and Ill-posed Problems 18}, 6 (2010),
  655--676.

\bibitem{lu2010model}
{\sc Lu, S., Pereverzev, S.~V., and Tautenhahn, U.}
\newblock A model function method in regularized total least squares.
\newblock {\em Applicable Analysis 89}, 11 (2010), 1693--1703.

\bibitem{lu2010regularized}
{\sc Lu, S., Pereverzev, S.~V., and Tautenhahn, U.}
\newblock Regularized total least squares: Computational aspects and error
  bounds.
\newblock {\em SIAM Journal on Matrix Analysis and Applications 31}, 3 (2010),
  918--941.

\bibitem{mhaskar2004polynomial}
{\sc Mhaskar, H.~N.}
\newblock Polynomial operators and local smoothness classes on the unit
  interval.
\newblock {\em Journal of Approximation Theory 131}, 2 (2004), 243--267.

\bibitem{mhaskar2005polynomial}
{\sc Mhaskar, H.~N.}
\newblock Polynomial operators and local smoothness classes on the unit
  interval, ii.
\newblock {\em Jaen Journal on Approximation Theorem 1\/} (2005), 1--25.

\bibitem{mhaskar2020direct}
{\sc Mhaskar, H.~N.}
\newblock A direct approach for function approximation on data defined
  manifolds.
\newblock {\em Neural Networks 132\/} (2020), 253--268.

\bibitem{mhaskar2001spherical}
{\sc Mhaskar, H.~N., Narcowich, F.~J., and Ward, J.~D.}
\newblock Spherical {M}arcinkiewicz--{Z}ygmund inequalities and positive
  quadrature.
\newblock {\em Mathematics of Computation 70}, 235 (2001), 1113--1130.

\bibitem{mhaskar2013filtered}
{\sc Mhaskar, H.~N., Naumova, V., and Pereverzyev, S.~V.}
\newblock Filtered {L}egendre expansion method for numerical differentiation at
  the boundary point with application to blood glucose predictions.
\newblock {\em Applied Mathematics and Computation 224\/} (2013), 835--847.

\bibitem{mhaskar2017deep}
{\sc Mhaskar, H.~N., Pereverzyev, S.~V., and van~der Walt, M.~D.}
\newblock A deep learning approach to diabetic blood glucose prediction.
\newblock {\em Frontiers in Applied Mathematics and Statistics 3\/} (2017), 14.

\bibitem{pereverzyev2015parameter}
{\sc Pereverzyev, S.~V., Sloan, I.~H., and Tkachenko, P.}
\newblock Parameter choice strategies for least-squares approximation of noisy
  smooth functions on the sphere.
\newblock {\em SIAM Journal on Numerical Analysis 53}, 2 (2015), 820--835.

\bibitem{rivlin1980introduction}
{\sc Rivlin, T.~J.}
\newblock {\em An introduction to the approximation of functions}.
\newblock Blaisdell, Waltham, Massachusetts, 1969.

\bibitem{rutishauser1990lectures}
{\sc Rutishauser, H.}
\newblock {\em Lectures on numerical mathematics}.
\newblock Birkh\"{a}user, Berlin, 1990.

\bibitem{salzer1972lagrangian}
{\sc Salzer, H.~E.}
\newblock {L}agrangian interpolation at the {C}hebyshev points
  $x_{n,\nu}\equiv\cos(\nu\pi/n)$, $\nu={O}(1)n$; some unnoted advantages.
\newblock {\em The Computer Journal 15}, 2 (1972), 156--159.

\bibitem{schwarz1989numerical}
{\sc Schwarz, H.~R., and Waldvogel, J.}
\newblock {\em Numerical analysis: a comprehensive introduction}.
\newblock Wiley, New York, 1989.

\bibitem{sloan1995polynomial}
{\sc Sloan, I.~H.}
\newblock Polynomial interpolation and hyperinterpolation over general regions.
\newblock {\em Journal of Approximation Theory 83}, 2 (1995), 238--254.

\bibitem{von1939orthogonal}
{\sc Szeg{\H{o}}, G.}
\newblock {\em Orthogonal polynomials}, vol.~23 of {\em Colloquium Publications
  Volume XXIII}.
\newblock American Mathematical Society, Providence, Rhode Island, 1939.

\bibitem{tikhonov1977solutions}
{\sc Tikhonov, A.~N., and Arsenin, V.~J.}
\newblock {\em Solutions of ill-posed problems}.
\newblock Winston \& Sons, Washington, D.C., 1977.

\bibitem{trefethen2013approximation}
{\sc Trefethen, L.~N.}
\newblock {\em Approximation theory and approximation practice}, vol.~128.
\newblock SIAM, Philadelphia, 2013.

\bibitem{wang2014explicit}
{\sc Wang, H., Huybrechs, D., and Vandewalle, S.}
\newblock Explicit barycentric weights for polynomial interpolation in the
  roots or extrema of classical orthogonal polynomials.
\newblock {\em Mathematics of Computation 83}, 290 (2014), 2893--2914.

\bibitem{wang2012convergence}
{\sc Wang, H., and Xiang, S.}
\newblock On the convergence rates of {L}egendre approximation.
\newblock {\em Mathematics of Computation 81}, 278 (2012), 861--877.

\bibitem{wei2016tikhonov}
{\sc Wei, Y., Xie, P., and Zhang, L.}
\newblock Tikhonov regularization and randomized {GSVD}.
\newblock {\em SIAM Journal on Matrix Analysis and Applications 37}, 2 (2016),
  649--675.

\bibitem{xiang2013regularization}
{\sc {Xiang}, H., and {Zou}, J.}
\newblock {Regularization with randomized SVD for large-scale discrete inverse
  problems.}
\newblock {\em Inverse Problems 29}, 8 (2013), 085008.

\bibitem{xiang2015randomized}
{\sc Xiang, H., and Zou, J.}
\newblock {Randomized algorithms for large-scale inverse problems with general
  Tikhonov regularizations.}
\newblock {\em Inverse Problems 31}, 8 (2015), 24.

\bibitem{xiang2012error}
{\sc Xiang, S.}
\newblock On error bounds for orthogonal polynomial expansions and {G}auss-type
  quadrature.
\newblock {\em SIAM Journal on Numerical Analysis 50}, 3 (2012), 1240--1263.

\bibitem{zhong2012multiscale}
{\sc Zhong, M., Lu, S., and Cheng, J.}
\newblock Multiscale analysis for ill-posed problems with semi-discrete
  tikhonov regularization.
\newblock {\em Inverse Problems 28}, 6 (2012), 065019.

\end{thebibliography}

\end{document}